\documentclass[12pt]{article}

\usepackage{amsmath,amsthm}
\usepackage{amssymb}

\usepackage{enumerate}

\newtheorem{thm}{Theorem}[section]
\newtheorem{cor}[thm]{Corollary}
\newtheorem{lem}[thm]{Lemma}

\theoremstyle{definition}

\newtheorem{rem}[thm]{Remark}

\numberwithin{equation}{section}

\begin{document}

%%%%% To ease editing, for IMPAN journals add:

%\baselineskip=17pt

%%%%%%%%%%%%%%%%

\title{On rotational surfaces in pseudo--Euclidean space $\mathbb E^4_t$ with pointwise 1--type Gauss map}

\author{Burcu Bekta\c{s}, Elif \"{O}zkara Canfes and U\u{g}ur Dursun}

\date{}

\maketitle

%% Classification and key words; note that the 2010 classification is used:

\renewcommand{\thefootnote}{}

\footnote{2010 \emph{Mathematics Subject Classification}: Primary 53B25; Secondary 53C50.}

\footnote{\emph{Key words and phrases}: Pointwise 1--type Gauss map, rotational surfaces, 
parallel mean curvature vector,  normal bundle, zero mean curvature vector.}

\renewcommand{\thefootnote}{\arabic{footnote}}
\setcounter{footnote}{0}

%%%%%%%%

\begin{abstract}
In this work, we study some classes of rotational surfaces 
in the pseudo--Euclidean space $\mathbb E^4_t$ with profile curves 
lying in 2--dimensional planes. 
First, we determine all such surfaces in the Minkowski 4--space $\mathbb{E}^4_1$ 
with pointwise 1--type Gauss map of the first kind and second kind. 
Then, we obtain rotational surfaces in $\mathbb{E}^4_2$ with zero mean curvature 
and having pointwise 1--type Gauss map of second kind.
\end{abstract}

\section{Introduction}
In late 1970, B.-Y. Chen introduced the concept of finite type submanifolds of Euclidean space, \cite{CH1}. 
Since then many works have been done to characterize or classify submanifolds of Euclidean space 
or pseudo--Euclidean space in terms of finite type. 
Then the notion of finite type was extended to differentiable maps, in particular Gauss map of 
submanifolds by B.-Y. Chen and P. Piccinni, \cite{Chen-Piccinni}. 
A smooth map $\phi$ on a submanifold $M$ of a Euclidean space or a pseudo Euclidean space is 
said to be {\it finite type} if $\phi$ has a finite spectral resolution, that is, 
$\phi=\phi_0+\sum_{t=1}^k \phi_t$, where $\phi_0$ is a constant vector and $\phi_t$'s are 
non-constant maps such that $\Delta\phi_{t}=\lambda_{t}\phi_t,\; \lambda_t\in\mathbb{R}, \;t=1, 2, \cdots, k$. 

If a submanifold $M$ of a Euclidean space or a pseudo--Euclidean space has 1--type Gauss map $\nu$, 
then $\nu$ satisfies $\Delta\nu=\lambda(\nu+C)$ for some $\lambda\in\mathbb{R}$ and for some constant vector $C$.  
Also, it has been seen that the equation 
\begin{equation}
\label{DefinitonPW1type}
\Delta \nu =f(\nu +C)
\end{equation}
is satisfied for some smooth function $f$ on $M$ and some constant vector $C$ by the Gauss map of 
some submanifolds such as helicoid, catenoid, right cones in $\mathbb{E}^3$ and Enneper's 
hypersurfaces in $\mathbb E^{n+1}_1$, \cite{UDur2,Kim-Yoon}.
A submanifold of a Euclidean or a pseudo--Euclidean space is said to have pointwise 1--type Gauss map 
if it satisfies \eqref{DefinitonPW1type}.
A submanifold with pointwise 1--type Gauss map is said to be of {\it the first kind} 
if $C$ is the zero vector. Otherwise, it is said to be of {\it the second kind}. 
\begin{rem}
\label{planeGaussmap}
For an n--dimensional plane $M$ in a pseudo--Euclidean space, the Gauss map $\nu$ is constant
and $\Delta\nu=0$. For $f=0$ if we write $\Delta\nu=0.\nu$, then $M$ has pointwise 1--type 
Gauss map of the first kind. If we choose $C=-\nu$ for any nonzero smooth function $f$,
then \eqref{DefinitonPW1type} holds. In this case, $M$ has pointwise 1--type Gauss map
of the second kind. Therefore we say that an n--dimensional plane $M$ in a pseudo--Euclidean
space is a trivial pseudo--Riemannian submanifold with pointwise 1--type Gauss map of 
the first kind and the second kind.  
\end{rem}
The classification of ruled surfaces and rational surfaces in $\mathbb E^3_1$ with 
pointwise 1--type Gauss map were studied in \cite{Choi-Kim-YoonDW2011,KKKM}. 
Also, in \cite{UDur2} and \cite{Kim-Yoon-3}, a characterization of rotational hypersurface and a complete classification of cylindrical and non--cylindrical surfaces in $\mathbb{E}^m_1$ were obtained, respectively. 

The complete classification of Vranceanu rotational surfaces in the pseudo--Euclidean 
$\mathbb{E}^4_2$ with pointwise 1-type Gauss map was obtained in \cite{Kim-Yoon-2}, 
and it was proved that a flat rotational surface in $\mathbb{E}^4_2$ with pointwise 1--type Gauss map 
is either the product of two plane hyperbolas or the product of a plane circle and a plane hyperbola. 

Recently,  a classification of flat spacelike and 
timelike rotational surfaces in $\mathbb{E}^4_1$ with pointwise 1--type Gauss map were given 
\cite{Dursun-BB1, Dursun-BB}.

In this article, we present some results on rotational surfaces 
in the pseudo--Euclidean space $\mathbb E^4_t$ with profile curves lying in 2--dimensional 
planes and having pointwise 1--type Gauss map.  
First, we give classification of all such surfaces in the Minkowski space $\mathbb{E}^4_1$ defined by 
\eqref{CiftDonelYuzeyMnfld}, called double rotational surface, with pointwise 1-type Gauss map 
of the first kind. 
Then, we show that there exists no a non-planar timelike double rotational surface in $\mathbb{E}^4_1$  with 
flat normal bundle and pointwise 1-type Gauss map of the second kind. Finally, 
we determine the rotational surfaces in the pseudo-Euclidean  $\mathbb{E}^4_2$ defined by \eqref{1DY} and \eqref{2DY} 
with zero mean curvature and pointwise 1-type Gauss map of the second kind.

\section{Preliminaries} 
\label{prelim}
Let $\mathbb E^m_t$ denote  $m$-dimensional pseudo--Euclidean space with the canonical
metric given by
$$
g=\sum_{i=1}^{m-t} dx_i^2-\sum_{j=m-t+1}^{m} dx_j^2,
$$
where $(x_1, x_2, \hdots, x_m)$  is a rectangular coordinate system in $\mathbb E^m_t$.

\noindent We put 
\begin{eqnarray}
\label{ykürehiper}
\mathbb{S}^{m-1}_t(x_0,r^2)=\left\{x\in\mathbb{E}_t^{m}\;|\; \langle x-x_0, x-x_0\rangle= r^{-2}\right\},\\
\mathbb{H}^{m-1}_{t-1}(x_0,-r^2)=\left\{x\in\mathbb{E}_{t}^{m}\;|\; \langle x-x_0, x-x_0\rangle=-r^{-2}\right\},
\end{eqnarray}
where $\langle , \rangle$ is the indefinite inner product associated to $g$. Then $\mathbb{S}^{m-1}_t(x_0,r^2)$ and $\mathbb{H}^{m-1}_{t-1}(x_0,-r^2)$ are complete pseudo--Riemannian manifolds of constant curvature $r^2$ and $-r^2$, respectively. We denote $\mathbb{S}^{m-1}_t(x_0,r^2)$ and $\mathbb{H}^{m-1}_{t-1}(x_0,-r^2)$ by $\mathbb{S}^{m-1}_t(r^2)$ and $\mathbb{H}^{m-1}_{t-1}(-r^2)$ when $x_0$ is the origin.
In particular, $\mathbb{E}^m_1$, $\mathbb{S}^{m-1}_1(x_0,r^2)$ and $\mathbb{H}^{m-1}_1(x_0,-r^2)$ are known as \textit{the Minkowski}, \textit{de Sitter}, and \textit{anti-de Sitter spaces}, respectively.

A vector $v \in \mathbb{E}^m_t$ is called spacelike (resp., timelike)
if $\langle v, v\rangle>0$ or $v=0$ (resp., $\langle v,v\rangle<0$). 
A vector $v$ is called lightlike if $\langle v, v\rangle=0$, and $v\neq 0$.

Let $M$ be an oriented $n$--dimensional pseudo--Riemannian submanifold in an $m$--dimensional
pseudo--Euclidean space $\mathbb E^m_t$. 
We choose an oriented local orthonormal  frame $\{e_1, \dots, $ $ e_{m} \}$
on $M$ with $\varepsilon_A= \langle e_A, e_A \rangle = \pm 1$ such that  $e_1,\dots,e_{n}$ are
tangent to $M$ and $e_{n+1}, \dots, e_{m}$ are normal to $M$.
We use the following convention on the range of indices:
$1\leq i,j,k,\ldots \leq n, \; n+1\leq r,s, t,\ldots \leq m$.

Let $\widetilde{\nabla}$ be the Levi--Civita connection of $\mathbb E^m_t$
and  $\nabla$  the induced connection on $M$. Denote by $\{\omega^1,\dots,\omega^{m} \}\, $
the dual frame  and by  $\{\omega_{AB}\},  A,B=1,\dots, m$, the
connection forms associated to $\{e_1,\dots,e_{m} \}$. Then we have
%%%%%%%%%%%
\begin{align}
\widetilde{\nabla}_{e_k}e_i= &\sum_{j=1}^{n}\varepsilon_j \omega_{ij}(e_k) e_j + \sum_{r=n+1}^{m} \varepsilon_r h^r_{ik} e_r, \notag\\
\widetilde{\nabla}_{e_k}e_r= & - A_r(e_k)+ \sum_{s=n+1}^{m} \varepsilon_s \omega_{rs}(e_k) e_s,\;\; 
D_{e_k}e_r= \sum_{s=n+1}^{m} \varepsilon_s \omega_{rs}(e_k) e_s, \notag
\end{align}
where $D$ is the normal connection, $h^r_{ij}$ the coefficients of
the second fundamental form $h$, and
$A_r$ the Weingarten map in the direction $e_r$.

The mean curvature vector $H$ and the squared length $\| h\|^2$ of
the second fundamental form $h$ are defined, respectively, by
\begin{equation} 
\label{meancurvature}
H= \frac{1}{n}\sum_{i,r} \varepsilon_i\varepsilon_r h^r_{ii} e_r
\end{equation}
and
\begin{equation} 
\label{sqr-length}
\| h\|^2 = \sum_{i,j,r} \varepsilon_i\varepsilon_j \varepsilon_r  h^r_{ij}h^r_{ji}.
\end{equation}
A submanifold $M$ is said to have parallel mean curvature vector $H$
if $DH =0$ identically.

The gradient of a smooth function $f$ on $M$ is defined by 
$\nabla f=\sum\limits_{i=1}^n  \varepsilon_i e_i(f)e_i,$
and the Laplace operator acting on $M$ is 
$\Delta =\sum\limits_{i=1}^n\varepsilon_i(\nabla_{e_i}e_i-e_ie_i).$

The Codazzi equation of $M$ in $\mathbb{E}^m_t$ is given by
\begin{align} 
\label{codazzi}
\begin{split}
&h^r_{ij,k}= h^r_{jk,i},\\
&h^r_{jk,i} = e_i(h^r_{jk}) -
\sum_{\ell=1}^{n} \varepsilon_{\ell}\left (  h^r_{\ell k} \omega_{j\ell}(e_i) +
h^r_{\ell j} \omega_{k \ell}(e_i) \right ) + \sum_{s=n+1}^{m} \varepsilon_s h^s_{jk} \omega_{sr} (e_i).
\end{split}
\end{align}
Also, from the Ricci equation of $M$ in $\mathbb{E}^{m}_t$, we have
\begin{equation} 
\label{normal-curv}
R^D(e_j, e_k; e_r, e_s) = \langle [A_r,A_s](e_j), e_k \rangle =
\sum_{i=1}^{n} \varepsilon_i \left (h^r_{ik} h^s_{ij} - h^r_{ij} h^s_{ik}\right),
\end{equation}
where $R^D$ is the normal curvature tensor.

A submanifold $M$ in $\mathbb E^m_t$ is said to have flat normal bundle 
if $R^D$ vanishes identically.

Let $G(m-n, m)$ be the Grassmannian manifold consisting of
all oriented $(m-n)$--planes through the origin of an $m$--dimensional pseudo--Euclidean space  
$\mathbb E^m_t$ with index $t$
and $\bigwedge^{m-n} \mathbb E^m_t$  the vector space obtained by
the exterior product of  $m-n$ vectors in $\mathbb E^m_t$. 
Let  $f_{i_1} \wedge \cdots \wedge f_{i_{m-n}}$ and $g_{i_1}
\wedge \cdots \wedge g_{i_{m-n}}$ be two vectors in $\bigwedge^{m-n} \mathbb E^m_t$,
where  $\{f_1, f_2, \dots, f_m\}$ and $\{g_1, g_2, \dots, g_m\}$
are two orthonormal bases of $\mathbb E^m_t$.
Define an indefinite inner product $\langle\langle,  \rangle\rangle$ on
$\bigwedge^{m-n} \mathbb E^m_t$ by
\begin{equation} 
\label{inner-prod}
\langle\langle  f_{i_1} \wedge \cdots \wedge f_{i_{m-n}},  g_{i_1} \wedge \cdots \wedge g_{i_{m-n}} \rangle\rangle
= \det( \left\langle f_{i_\ell}, g_{j_k}\right\rangle).
\end{equation}
Therefore, for some positive integer $s$, we may identify $\bigwedge^{m-n} \mathbb E^m_t$ with
some pseudo--Euclidean space $\mathbb E^N_s$,
where $N= \binom{m}{m-n}$. 
The map $\nu : M \rightarrow G(m-n, m) \subset   \mathbb E^N_s$
from an oriented pseudo-Riemannian submanifold $M$ into
$G(m-n, m)$ defined by
\begin{equation}\label{MinkGaussTasvTanim}
\begin{array}{rcl}
\nu(p) = (e_{n+1} \wedge e_{n+2} \wedge \cdots \wedge e_{m}) (p)
\end{array}
\end{equation}
is called the {\it Gauss map} of $M$ which assigns to a point $p$ in $M$  
the oriented $(m-n)$--plane through the origin of $\mathbb E^m_t$ and parallel
to the normal space of $M$ at $p$, \cite{Kim-Yoon-2}.

We put $\varepsilon = \langle\langle \nu, \nu \rangle\rangle =
\varepsilon_{n+1} \varepsilon_{n+2} \cdots \varepsilon_m = \pm 1$ and
\begin{equation}
\label{l-like-n2-cmc}
\widetilde M^{N-1}_s (\varepsilon) =\left\{ \displaystyle
 \begin{array}{lll}
\displaystyle \mathbb S^{N-1}_s (1) \;\;& \mbox{in} \; \;\mathbb E^N_s & \mbox{if} \; \; \varepsilon =1
  \\ \notag
%%%%%%%%%%%%%
\displaystyle \mathbb H^{N-1}_{s-1} (-1) \;\;& \mbox{in} \;
 \; \mathbb E^N_s & \mbox{if} \; \; \varepsilon=-1.
\end{array}
\right.
\end{equation}
Then the Gauss image $\nu(M)$ can be viewed as
$\nu(M) \subset \widetilde M^{N-1}_s (\varepsilon)$.

\subsection{Rotational surfaces in $\mathbb{E}^4_1$ with profile curves lying in 2-planes}
We consider timelike rotational surfaces in the Minkowski space $\mathbb{E}^4_1$ 
whose profile curves lie in timelike 2--planes. By choosing a profile curve 
$\gamma(s)=(x(s),0,0,w(s))$ in the $xw$--plane defined on an open interval $I$ in $\mathbb{R}$. 
We can parametrize a timelike rotational surface in $\mathbb{E}^4_1$ as follows
\begin{equation}\label{CiftDonelYuzeyMnfld}
M:r(s,t)=(x(s)\cos at, x(s)\sin at, w(s)\sinh bt, w(s)\cosh bt),
\end{equation}
where $s$ is the arc lenght parameter of $\gamma$, $s\in\mathbb{R}$ and 
$t \in (0, 2\pi)$. The rotational surface $M$ is called a double rotational surface in $\mathbb{E}^4_1$. 
Then, ${x'}^2(s) - {w'}^2(s) =-1$ and the curvature
$\kappa$ of $\gamma$ is given by $\kappa(s)=w'(s)x''(s)-x'(s)w''(s)$.

We form the following orthonormal moving frame field $\{e_1, e_2, e_3, e_4 \}$
on $M$ such that $e_1, e_2$ are tangent to $M$, and  $e_3, e_4 $ are normal to $M$:
\begin{eqnarray}
\label{TegNormVektorsCiftDY1} e_1&=&\frac{\partial}{\partial s}, \quad
e_2 =\frac{1}{q}\frac{\partial}{\partial t},\\
\label{TegNormVektorsCiftDY3} e_3&=&(w'(s)\cos at, w'(s)\sin at, x'(s)\sinh bt, x'(s)\cosh bt),\\
\label{TegNormVektorsCiftDY4} e_4&=&\frac{1}{q}(bw(s)\sin at, -bw(s)\cos at, ax(s)\cosh bt, ax(s)\sinh bt),
\end{eqnarray}
where $q= \sqrt{a^2 x^2(s) + b^2 w^2(s)}$
and $\varepsilon_1 = -1, \; \varepsilon_2 = \varepsilon_3 = \varepsilon_4 = 1$.

By a direct computation, we have the components of the second fundamental form and 
the connection forms as follows
\begin{align}
\label{CDYuzeyTemelForm3}
h^3_{11}&=\kappa(s),    &h^3_{22}=-\frac{a^2x(s)w'(s) + b^2w(s)x'(s)}{a^2x^2(s)+b^2w^2(s)},  \\
\label{CDYuzeyTemelForm4}
h^3_{12}&=h^4_{11}= h^4_{22}=0,  &h^4_{12}=\frac{ab(x(s)w'(s)-w(s)x'(s))}{a^2x^2(s)+b^2w^2(s)}, \\
\label{CDYuzeyTegetKonForm}
\omega_{12}(e_1)&=0,  &\omega_{12}(e_2)=\frac{a^2x(s)x'(s)+b^2w(s)w'(s)}{a^2x^2(s)+b^2w^2(s)}, \\
\label{CDYuzeyNormlKonForm}
\omega_{34}(e_1)&=0,  &\omega_{34}(e_2)=\frac{ab(x(s)x'(s)-w(s)w'(s))}{a^2x^2(s)+b^2w^2(s)}.
\end{align}

Hence we obtain the mean curvature vector and the normal curvature of $M$ from $\eqref{meancurvature}$ 
and $\eqref{normal-curv}$, respectively, as 
\begin{eqnarray}
\label{CiftDonYuzOrtEgrVek} H&=&\frac 12(h^3_{22}- h^3_{11})e_3,\\
\label{CiftDonYuzNormEgrilik} R^D(e_1,e_2;e_3,e_4)&=&h^4_{12}(h^3_{11}+h^3_{22}).
\end{eqnarray}

On the other hand, from the Codazzi equation \eqref{codazzi} we have
\begin{eqnarray}
\label{CDYUzCodazzi1} e_1(h^3_{22})&=&-\omega_{12}(e_2)\left(h^3_{11}
+ h^3_{22} \right) - h^4_{12}\omega_{34}(e_2),\\
\label{CDYUzCodazzi2}e_1(h^4_{12})&=&-2 h^4_{12} \omega_{12}(e_2)
 +  h^3_{11}\omega_{34}(e_2).
\end{eqnarray}  
  
%%%%%%%%%%%%%%%%%%%%%%%%%%%%%%%%%%%%%%%%%%%%%%%

\subsection{Rotational surfaces in $\mathbb{E}^4_2$ with profile curves lying in 2--planes}
In the pseudo-Euclidean space $\mathbb{E}^4_2$, we consider two rotational surfaces 
whose profile curves lie in 2--planes.

First, we choose a profile curve $\alpha$ in the $yw$--plane as 
$\alpha(s)=(0, y(s), 0, w(s))$ defined on an open interval $I\subset\mathbb{R}$.
Then the parametrization of the rotational surface $M_1(b)$ in $\mathbb{E}^4_2$ is given by
\begin{equation}\label{1DY}
M_1(b): r_1(s,t)=(w(s)\sinh t, y(s)\cosh (bt), y(s)\sinh (bt), w(s)\cosh t), 
\end{equation}
for some constant $b>0$, where $s\in I$ and $t\in\mathbb{R}$.

Secondly, we choose a profile curve $\beta$ in the $xz$--plane as 
$\beta(s)=(x(s),0,z(s),0)$ defined on an open interval $I\subset\mathbb{R}$. 
Then the parametrization of the rotational surface $M_2(b)$ in $\mathbb{E}^4_2$ is given by
\begin{equation}\label{2DY}
M_2(b): r_2(s,t)=(x(s)\cos t, x(s)\sin t, z(s)\cos (bt), z(s)\sin (bt)), 
\end{equation}
for some constant $b>0$, where $s\in I$ and $t\in (0, 2\pi)$.

Now, for the rotational surface $M_1(b)$ defined by \eqref{1DY}, we consider the 
following orthonormal moving frame field $\{e_1, e_2, e_3, e_4\}$ on $M_1(b)$ 
such that $e_1, e_2$ are tangent to $M_1(b)$, and $e_3, e_4$ are normal to $M_1(b)$:
\begin{eqnarray}
\label{TegNormVektors1DY} e_1&=&\frac{1}{q}\frac{\partial}{\partial t}, \quad
e_2 =\frac{1}{A}\frac{\partial}{\partial s},\\
\label{TegNormVektors1DY3} e_3&=&\frac{1}{A}(y'(s)\sinh t, w'(s)\cosh (bt), w'(s)\sinh (bt), y'(s)\cosh t),\\
\label{TegNormVektors1DY4} e_4&=&-\frac{\varepsilon\varepsilon^*}{q}(by(s)\cosh t, w(s)\sinh (bt), w(s)\cosh (bt), by(s)\sinh t),
\end{eqnarray}
where $A=\sqrt{\varepsilon({y^\prime}^2(s)-{w^\prime}^2(s))}\neq 0$, $q= \sqrt{\varepsilon^*(w^2(s)- b^2 y^2(s))}\neq 0$, and 
$\varepsilon=\mbox{sgn}{({y^\prime}^2(s)-{w^\prime}^2(s))}$, $\varepsilon^*=\mbox{sgn}(w^2(s)- b^2 y^2(s))$.
Then, $\varepsilon_1 = -\varepsilon_4=\varepsilon^* , \; \varepsilon_2 = -\varepsilon_3 =\varepsilon$.

By a direct calculation, we have the components of the second fundamental form 
and the connection forms as follows
\begin{align}
\label{1DYuzeyTemelForm3}
h^3_{11}&=\frac{1}{Aq^2}(b^2y(s)w'(s)-w(s)y'(s)),\quad h^3_{22}=\frac{1}{A^3}(w'(s)y''(s)-y'(s)w''(s)),   \\
\label{1DYuzeyTemelForm4}
h^4_{12}&=\frac{\varepsilon\varepsilon^* b}{Aq^2}(w(s)y'(s)-y(s)w'(s)), \quad h^3_{12}=h^4_{11}=h^4_{22}=0, \\
\label{1DYuzeyTegetKonForm}
\omega_{12}(e_1)&= \frac{1}{Aq^2}(b^2y(s)y'(s)-w(s)w'(s)), \quad \omega_{12}(e_2)=0, \\
\label{1DYuzeyNormlKonForm}
\omega_{34}(e_1)&= \frac{\varepsilon\varepsilon^* b}{Aq^2}(w(s)w'(s)-y(s)y'(s)),\quad \omega_{34}(e_2)=0.
\end{align}

Similarly, for the rotational surface $M_2(b)$ defined by \eqref{2DY},
we consider the following orthonormal moving frame field $\{e_1, e_2, e_3, e_4\}$ on $M_2(b)$ 
such that $e_1, e_2$ are tangent to $M_2(b)$, and $e_3, e_4$ are normal to $M_2(b)$:
\begin{eqnarray}
\label{TegNormVektors2DY} e_1&=&\frac{1}{\bar{q}}\frac{\partial}{\partial t}, \quad
e_2 =\frac{1}{\bar{A}}\frac{\partial}{\partial s},\\
\label{TegNormVektors2DY3} e_3&=&\frac{1}{\bar{A}}(z'(s)\cos t, z'(s)\sin t, x'(s)\cos (bt), x'(s)\sin (bt)),\\
\label{TegNormVektors2DY4} e_4&=&-\frac{\varepsilon\varepsilon^*}{\bar{q}}(bz(s)\sin t, -bz(s)\cos t, x(s)\sin (bt), -x(s)\cos (bt)),
\end{eqnarray}
where $\bar{A}=\sqrt{\varepsilon({x^\prime}^2(s)-{z^\prime}^2(s))}\neq 0$, 
$\bar{q}= \sqrt{\varepsilon^*(x^2(s)- b^2 z^2(s))}\neq 0$,   
$\varepsilon=\mbox{sgn}{({x^\prime}^2(s)-{z^\prime}^2(s))}$, 
and  $\varepsilon^*=\mbox{sgn}(x^2(s)- b^2 z^2(s))$.
Then, $\varepsilon_1 = -\varepsilon_4=\varepsilon^* , \; \varepsilon_2 = -\varepsilon_3 =\varepsilon$.

By a direct computation, we have the components of the second fundamental form 
and the connection forms as follows
\begin{align}
\label{2DYuzeyTemelForm3}
h^3_{11}&=\frac{1}{\bar{A}{\bar{q}}^2}(b^2z(s)x'(s)-x(s)z'(s)),\quad h^3_{22}=\frac{1}{{\bar{A}}^3}(z'(s)x''(s)-x'(s)z''(s)),   \\
\label{2DYuzeyTemelForm4}
h^4_{12}&=\frac{\varepsilon\varepsilon^* b}{\bar{A}{\bar{q}}^2}(z(s)x'(s)-x(s)z'(s)), \quad h^3_{12}=h^4_{11}=h^4_{22}=0, \\
\label{2DYuzeyTegetKonForm}
\omega_{12}(e_1)&= \frac{1}{\bar{A}{\bar{q}}^2}(b^2z(s)z'(s)-x(s)x'(s)), \quad \omega_{12}(e_2)=0, \\
\label{2DYuzeyNormlKonForm}
\omega_{34}(e_1)&= \frac{\varepsilon\varepsilon^* b}{\bar{A}{\bar{q}}^2}(z(s)z'(s)-x(s)x'(s)),\quad \omega_{34}(e_2)=0.
\end{align}

Therefore, we have the mean curvature vector and normal curvature 
for the rotational surfaces $M_1(b)$ and $M_2(b)$ as follows
\begin{eqnarray}
\label{12DYOrtEgrVek} H&=&-\frac 12(\varepsilon\varepsilon^* h^3_{11} + h^3_{22})e_3,\\
%\label{12DYGaussEgr}  K&=& \varepsilon^*(h_{12}^4)^2-\varepsilon h_{11}^3 h_{22}^3,\\
\label{12DYNormEgrilik} R^D(e_1,e_2;e_3,e_4)&=&h^4_{12}(\varepsilon h^3_{22}-\varepsilon^* h^3_{11}).
\end{eqnarray}

On the other hand, by using the Codazzi equation \eqref{codazzi} we obtain
\begin{eqnarray}
\label{12DYUzCodazzi1} e_2(h^3_{11})&=&\varepsilon^* h_{12}^4\omega_{34}(e_1)+\omega_{12}(e_1)(\varepsilon^* h_{11}^3-\varepsilon h_{22}^3),\\
\label{12DYUzCodazzi2}e_2(h^4_{12})&=&-\varepsilon h_{22}^3\omega_{34}(e_1) + 2 \varepsilon^*h_{12}^4\omega_{12}(e_1).
\end{eqnarray}

The rotational surfaces $M_1(b)$ and $M_2(b)$ defined by \eqref{1DY} and \eqref{2DY} 
for $b=1$, $x(s)=y(s)=f(s)\sinh s$ and $z(s)=w(s)=f(s)\cosh s$ are also known 
as Vranceanu rotational surface, where $f(s)$ is a smooth function, \cite{HuiLiGuiLi}.

\section{Rotational surfaces in $\mathbb{E}^4_1$ with pointwise 1--type Gauss map} 
\label{sectionPW-1-kind}

In this section, we study rotational surfaces in the Minkowski 
space $\mathbb E^4_1$ defined by \eqref{CiftDonelYuzeyMnfld} with pointwise 1--type Gauss map.

%In \cite{Dursun-TurgayMink1stType}, 
By a direct calculation, the Laplacian of the Gauss map $\nu$ 
for an $n$--dimensional submanifold $M$ in a pseudo--Euclidean space
$\mathbb E^{n+2}_t$ is obtained as follows:

\begin{lem}\label{TeoremDursunArsan1}
%\cite{Dursun-TurgayMink1stType}  
Let $M$ be an $n$--dimensional submanifold of a pseudo--Euclidean space $\mathbb{E}^{n+2}_t$.
Then, the Laplacian of the Gauss map $\nu=e_{n+1}\wedge e_{n+2}$ is given by

\begin{align} \label{laplace-GaussMap}
\begin{split}
\Delta\nu =&||h||^2\nu
           +2\sum\limits_{j<k} \varepsilon_j \varepsilon_k R^D (e_j,e_k;e_{n+1},e_{n+2}) e_j\wedge e_k \\
           &+ \nabla(\mathrm{tr} A_{n+1})\wedge  e_{n+2}  + e_{n+1}\wedge\nabla(\mathrm{tr} A_{n+2})  \\
           &+  n\sum\limits_{j=1}^n \varepsilon_j \omega_{(n+1)(n+2)}(e_j) H \wedge e_j,
\end{split}
\end{align}
where $||h||^2$ is the squared length of the second fundamental form, $R^D$
the normal curvature tensor, and $\nabla(\mathrm{tr} A_r)$ the gradient of $\mathrm{tr}A_r$.
\end{lem}

Let $M$ be a surface in the pseudo--Euclidean space $\mathbb E^4_t$. We choose a local orthonormal frame field
$\{e_1,e_2,e_3,e_4\}$ on $M$ such that  $e_1, e_2$ are tangent to $M$, and
$e_3, e_4$ are normal to $M$.
Let $C$ be a vector field in  $\Lambda^2\mathbb E^4_t\equiv \mathbb E^6_s$. Since the set
$\{e_A\wedge e_B|1\leq A < B\leq 4\}$  is an orthonormal basis for  $\mathbb E^6_s$, the vector  $C$ can be expressed as
\begin{equation} 
\label{constantvectorC}
C= \sum_{1\leq A < B \leq 4}\varepsilon_A\varepsilon_B C_{AB}\, e_A \wedge e_B,
\end{equation}
where $C_{AB} = \left\langle C, e_A \wedge e_B \right\rangle$.

\begin{lem}
\label{CVektoruSabitLemmaMinkowski}
A vector  $C$ in $\Lambda^2\mathbb E^4_t\equiv \mathbb E^6_s$ written by \eqref{constantvectorC} is
constant if and only if the following equations are satisfied for $i=1, 2$
\begin{align}
\label{MinkCVekSbtGYKos12}
e_i\left(C_{12}\right)=&\varepsilon_3h_{i2}^3C_{13}+ \varepsilon_4h_{i2}^4C_{14} -\varepsilon_3 h_{i1}^3C_{23}- \varepsilon_4h_{i1}^4C_{24},\\
\label{MinkCVekSbtGYKos13}
e_i\left(C_{13}\right)=&-\varepsilon_2h_{i2}^3C_{12}+\varepsilon_4\omega_{34}(e_i)C_{14}+\varepsilon_2
\omega_{12}(e_i)C_{23}-\varepsilon_4h_{i1}^4C_{34},\\
\label{MinkCVekSbtGYKos14}
e_i\left(C_{14}\right)=&-\varepsilon_2h_{i2}^4C_{12}-\varepsilon_3\omega_{34}(e_i)C_{13}+
\varepsilon_2\omega_{12}(e_i)C_{24}+\varepsilon_3h_{i1}^3C_{34},\\
\label{MinkCVekSbtGYKos23}
e_i\left(C_{23}\right)=&\varepsilon_1 h_{i1}^3C_{12}-\varepsilon_1\omega_{12}(e_i)C_{13}+
\varepsilon_4\omega_{34}(e_i)C_{24}-\varepsilon_4h_{i2}^4C_{34},\\
\label{MinkCVekSbtGYKos24}
e_i\left(C_{24}\right)=&\varepsilon_1h_{i1}^4C_{12}-\varepsilon_1\omega_{12}(e_i)C_{14}-
\varepsilon_3\omega_{34}(e_i)C_{23}+\varepsilon_3h_{i2}^3C_{34},\\
\label{MinkCVekSbtGYKos34}
e_i\left(C_{34}\right)=&\varepsilon_1h_{i1}^4C_{13}-\varepsilon_1h_{i1}^3
C_{14}+\varepsilon_2h_{i2}^4C_{23}-\varepsilon_2h_{i2}^3C_{24}.
\end{align}
\end{lem}

%As in \cite{Dursun-TurgayMink1stType}, 
Using \eqref{laplace-GaussMap} the following results can be stated for the characterization of
timelike surfaces in $\mathbb{E}^4_1$ with pointwise 1--type Gauss map of the first kind.

\begin{thm}
\label{MinMinimal1Cesit1tip}
Let $M$ be an oriented timelike surface with zero mean curvature in $\mathbb E^4_1$. 
Then $M$ has pointwise 1--type Gauss map of the first kind if and only if 
$M$ has flat normal bundle. Hence the Gauss map $\nu$ satisfies
\eqref{DefinitonPW1type} for $f=\|h\|^2$ and $C=0$.
\end{thm}

\begin{thm}
\label{NonmaksClassTheo}
Let $M$ be an oriented timelike surface with nonzero mean curvature in $\mathbb E^4_1$. 
Then $M$ has pointwise 1--type Gauss map of the first kind if and only if $M$
has parallel mean curvature vector.
\end{thm}

We will classify timelike rotational surface in $\mathbb{E}^4_1$ defined by
\eqref{CiftDonelYuzeyMnfld} with pointwise 1--type Gauss
map of the first kind by using the above theorems.

\begin{thm}
\label{NorBundFlatMinimal}
Let $M$ be a timelike rotational surface in $\mathbb{E}^4_1$ defined by
\eqref{CiftDonelYuzeyMnfld}.
Then $M$ has zero mean curvature, and its normal bundle is flat if and only if $M$
is an open part of a timelike plane in $\mathbb E^4_1$.
\end{thm}

\begin{proof}
Let $M$ be a timelike rotational surface in $\mathbb{E}^4_1$ given by 
\eqref{CiftDonelYuzeyMnfld}.
Then there exists a frame field $\{e_1, e_2, e_3, e_4\}$ defined on $M$ given by 
\eqref{TegNormVektorsCiftDY1}--\eqref{TegNormVektorsCiftDY4},
and the components of the second fundamental forms are given by \eqref{CDYuzeyTemelForm3} 
and \eqref{CDYuzeyTemelForm4}.
Since $M$ has zero mean curvature, and its normal bundle is flat,
then \eqref{CiftDonYuzOrtEgrVek} and \eqref{CiftDonYuzNormEgrilik} imply, respectively,
\begin{eqnarray}
\label{CiftDonYuzMinCond}h^3_{22}-\kappa&=&0,\\
\label{CiftDonYuzDuzNormCond} h^4_{12}(\kappa+h^3_{22})&=&0
\end{eqnarray}
as $h_{11}^3 = \kappa$, where $\kappa$ is the curvature of the profile curve of $M$.
By using \eqref{CiftDonYuzMinCond} and \eqref{CiftDonYuzDuzNormCond} 
we obtain $h^4_{12}\kappa=0$ which implies either $\kappa=0$ or $h^4_{12}=0$.

\noindent {\bf Case 1.} $\kappa=0$. Then the profile curve of $M$ is a line. 
We can parametrize the line as
\begin{eqnarray} \label{CDProfEgrBirDogru}
x(s)= x_0s+x_1, \quad w(s)= w_0s+w_1
\end{eqnarray}
for some constants $x_0,\ x_1,\ w_0,\ w_1\in\mathbb{R}$ with $x_0^2-w_0^2=-1$.
From \eqref{CiftDonYuzMinCond} we also have $h^3_{22}=0$.
By using the second equation in \eqref{CDYuzeyTemelForm3}
and \eqref{CDProfEgrBirDogru} we obtain
$$h^3_{22} = -\frac{ (a^2 + b^2)x_0w_0 s+ a^2x_1w_0 + b^2x_0w_1}{a^2(x_0s+x_1)^2+b^2(w_0s+w_1)^2}=0$$
which gives
\begin{eqnarray}
\label{CebDenkSisst1} (a^2 + b^2)x_0w_0=0,\\
\label{CebDenkSisst2} a^2x_1w_0 + b^2w_1x_0=0.
\end{eqnarray}
From \eqref{CebDenkSisst1} if $w_0=0$, then $x_0^2=-1$ which is inconsistent equation.
Hence, $w_0\neq 0$ and $x_0=0$, and thus $w_0=\pm 1$. Also, from \eqref{CebDenkSisst2}
we get $x_1 = 0$.
Thus, $x= 0$ which implies that $M$ is an open part of the timelike $zw$--plane.

\noindent \textbf{Case 2.} $h^4_{12}=0$. From the first equation 
in \eqref{CDYuzeyTemelForm4}
we have the differential equation $xw'-wx'=0$ that gives 
$x= c_0 w$ where $c_0$ is a constant.
Therefore, the profile curve of $M$ is an open part of a line passing through the origin.
Since the curvature $\kappa$ is zero, we have  $h_{11}^3 = 0$, and thus
$h_{22}^3 = 0$ because of \eqref{CiftDonYuzMinCond}.
From the second equation in \eqref{CDYuzeyTemelForm3} we get $c_0 (a^2+ b^2) w w' = 0$
which implies that $c_0 = 0$, i.e., $x=0$.
Therefore $M$ is an open part of the timelike $zw$--plane.\\
\indent In view of Remark \ref{planeGaussmap}, the converse of the proof is trivial.
\end{proof}

By Theorem \ref{MinMinimal1Cesit1tip} and  Theorem \ref{NorBundFlatMinimal}, we state

\begin{cor}
There exists no non--planar timelike surface with zero mean curvature in $\mathbb E^4_1$ defined by \eqref{CiftDonelYuzeyMnfld} with pointwise 1--type Gauss map of the first kind.
\end{cor}

Now, we focus on timelike rotational surfaces in $\mathbb{E}^4_1$ with
parallel nonzero mean curvature vector to obtain surfaces in $\mathbb{E}^4_1$ 
defined by \eqref{CiftDonelYuzeyMnfld} with
pointwise 1--type Gauss map of the first kind.

\begin{thm} 
\label{Rot-paralleH}
A timelike rotational surface in $\mathbb{E}^4_1$ defined by \eqref{CiftDonelYuzeyMnfld}
has parallel nonzero mean curvature vector if and only if it is an open part of the 
timelike surface defined by
\begin{align}
\label{CiftDonYuzOrtEgrParGYKosul}
\begin{split}
F(s,t)= (& r_0\cosh (\frac{s}{r_0})  \cos at, \;  r_0\cosh (\frac{s}{r_0})
 \sin at, \; r_0 \sinh (\frac{s}{r_0}) \sinh bt,  \\
 & r_0 \sinh(\frac{s}{r_0}) \cosh bt)
\end{split}
\end{align}
which has zero mean curvature in the de Sitter space 
$\mathbb{S}^3_1\left(\frac{1}{r_0^2}\right)\subset \mathbb{E}^4_1$.
\end{thm}

\begin{proof}
Let $M$ be a timelike rotational surface in $\mathbb E^4_1$ defined by \eqref{CiftDonelYuzeyMnfld}. 
Then, we have an orthonormal moving frame
$\{e_1, e_2, e_3, e_4\}$ on $M$ in
$\mathbb E^4_1$ given by \eqref{TegNormVektorsCiftDY1}--\eqref{TegNormVektorsCiftDY4},  
and the components of the second fundamental forms are given by \eqref{CDYuzeyTemelForm3} and \eqref{CDYuzeyTemelForm4}.
Suppose that the mean curvature vector $H$ is parallel, i.e., $ D_{e_i} H=0$ for $ i=1, 2$.
By considering \eqref{CDYuzeyNormlKonForm} and \eqref{CiftDonYuzOrtEgrVek} we have
%%%%%%%%
$$
D_{e_2} H  =   \frac{ab (h^3_{22}-h^3_{11}) (xx'-ww')}{2(a^2x^2+b^2w^2)}e_4=0.
$$
%%%%%%%%%%%%%%
Since $M$ has nonzero mean curvature, this equation reduces $xx'-ww'=0$ that implies 
$x^2 - w^2 = \mu_0$, where $\mu_0$ is a real number. 
Since $\gamma$ is a timelike curve with parametrized by arc length parameter $s$, 
we can choose $\mu_0=r_0^2$ and the components of $\gamma$ as
$$
x(s)=r_0\cosh \frac{s}{r_0},\quad w(s)=r_0\sinh \frac{s}{r_0}.
$$
Therefore, $M$ is an open part of the timelike surface given by \eqref{CiftDonYuzOrtEgrParGYKosul} 
which is minimal in the de Sitter space $\mathbb{S}^3_1\left(\frac{1}{r_0^2}\right)\subset\mathbb{E}^4_1$.\\
\indent The converse of the proof follows from a direct calculation.
\end{proof}

Considering Theorem \ref{NonmaksClassTheo} and Theorem \ref{Rot-paralleH} we state the following:
\begin{cor}
A timelike rotational surface $M$ with nonzero mean curvature in $\mathbb E^4_1$ defined by \eqref{CiftDonelYuzeyMnfld}
has pointwise 1--type Gauss map of the first kind if and only if
it is an open part of the surface given by \eqref{CiftDonYuzOrtEgrParGYKosul}.
\end{cor}

By combining \eqref{NorBundFlatMinimal} and \eqref{Rot-paralleH} 
we obtain the following classification theorem:
\begin{thm} \label{Class-PW1-kind}
Let $M$ be a timelike rotational surface in $\mathbb E^4_1$ defined by \eqref{CiftDonelYuzeyMnfld}.
Then $M$ has pointwise 1--type Gauss map of the first kind if and only if
$M$ is an open part of a timelike plane or the surface given by \eqref{CiftDonYuzOrtEgrParGYKosul}.
Moreover, the Gauss map $\nu = e_3 \wedge e_4$ of
the surface \eqref{CiftDonYuzOrtEgrParGYKosul} satisfies \eqref{DefinitonPW1type}
for $C = 0 $ and the function
$$
f=\| h\|^2 = \frac{2}{r_0^2}\left ( 1 - \frac{a^2b^2}{(a^2\cosh^2 (\frac{s}{r_0})+b^2\sinh^2 (\frac{s}{r_0}))^2} \right ).
$$
\end{thm}

Note that there is no non--planar timelike rotational surface in $\mathbb E^4_1$ defined by \eqref{CiftDonelYuzeyMnfld} with global 1--type Gauss map of the first kind.

Now, we investigate timelike rotational surfaces in $\mathbb E^4_1$ defined 
by \eqref{CiftDonelYuzeyMnfld} with pointwise 1--type Gauss map of the second kind.

\begin{thm} 
\label{Rot-surface-Gauss-sec-kind}
A timelike rotational surface $M$ in $\mathbb E^4_1$ defined by
\eqref{CiftDonelYuzeyMnfld} with flat normal bundle has pointwise 1--type Gauss map 
of the second kind if and only if
$M$ is an open part of a timelike plane in $\mathbb E^4_1$.
\end{thm}

\begin{proof}
Let $M$ be a timelike rotational surface with flat normal bundle in $\mathbb E^4_1$ defined by \eqref{CiftDonelYuzeyMnfld}.
Thus, we have
$ R^D(e_1, e_2; e_3, e_4) = h^4_{12}(h^3_{11} + h^3_{22}) =0$ which implies that
$h^4_{12}= 0 $ or $ h^3_{11} = -h^3_{22}\neq 0 $.

\noindent {\bf Case 1.} $h^4_{12}= 0 $. Now considering the second equation in 
\eqref{CDYuzeyTemelForm4} the general solution of $xw'-wx'=0$ is
$x= c_0 w$, where $c_0$ is constant. 
Hence, $M$ is a timelike regular cone in the Minkowski space $\mathbb E^4_1$. 
For $c_0= 0$, it can be easily seen that $M$ is an open part of the timelike
$zw$--plane. We suppose that $c_0\not =0$.
If we parametrize the line $x= c_0 w$ with respect to arc length parameter $s$, 
we then have
$w(s)= \pm\frac{1}{\sqrt{1- c_0^2}} s + w_0$ and 
$ x(s)=\pm\frac{c_0}{\sqrt{1- c_0^2}} s + c_0w_0, \; w_0, c_0 \in \mathbb{R}$ with $c_0^2<1$.
Thus, from \eqref{CDYuzeyTemelForm3}--\eqref{CDYuzeyNormlKonForm} we obtain that
\begin{align} 
\label{quantities-cone}
\begin{split}
h^3_{11}&=0, \quad \quad \; \; \, h^3_{22}=\mp\frac{c_0(a^2+ b^2)}{\sqrt{1- c_0^2}(a^2c_0^2 + b^2)w},\\
h^3_{12}&=0, \quad \quad \; \; \;  h^4_{ij}=0, \; i,j =1,2, \\
\omega_{12}(e_1)&= 0, \quad   \omega_{12}(e_2)=  \pm\frac{1}{\sqrt{1- c_0^2}w}, \\
\omega_{34}(e_1)&=  0, \quad   \omega_{34}(e_2)=\mp\frac{ab\sqrt{1- c_0^2}}{(a^2c_0^2 +  b^2)w }.
\end{split}
\end{align}
Therefore, using the equations \eqref{CDYUzCodazzi1} and \eqref{laplace-GaussMap}
the Laplacian of the Gauss map
$\nu = e_3\wedge e_4$ is given by
\begin{align} \label{laplace-GaussMap-cone}
\Delta\nu =&||h||^2\nu + h^3_{22} \omega_{12}(e_2) e_1\wedge e_4 -  h^3_{22}\omega_{34}(e_2) e_2\wedge e_3.
\end{align}

Assume that M has pointwise 1--type Gauss map of the second kind. 
Then there exists a smooth function $f$ and nonzero constant vector $C$ 
such that \eqref{DefinitonPW1type} is satisfied. 
Therefore, from \eqref{DefinitonPW1type} and \eqref{laplace-GaussMap-cone} we get
\begin{align}
\label{HHSbtRD0GaussTasCone1}
f(1+C_{34})=&\|h\|^2 = (h^3_{22})^2,\\
\label{HHSbtRD0GaussTasCone2}
fC_{14} =& - h^3_{22} \omega_{12}(e_2),\\
\label{HHSbtRD0GaussTasCone3}
fC_{23}=& - h^3_{22} \omega_{34}(e_2),\\
\label{HHSbtRD0GaussTasCone4}
C_{12}=&C_{13}=C_{24}=0.
\end{align}
It follows from \eqref{quantities-cone}, \eqref{HHSbtRD0GaussTasCone2}
and \eqref{HHSbtRD0GaussTasCone3} that $C_{14} \not =0$ and $C_{23} \not =0$.
%%%
Now, from \eqref{HHSbtRD0GaussTasCone2} and \eqref{HHSbtRD0GaussTasCone3} we have
\begin{equation} \label{NewC-eq-1}
\omega_{34}(e_2)C_{14} -  \omega_{12}(e_2)C_{23} =0.
\end{equation}

On the other hand, for $i=2$
equation \eqref{MinkCVekSbtGYKos13} implies
\begin{equation} \label{NewC-eq-2}
   \omega_{34}(e_2)C_{14} + \omega_{12}(e_2)C_{23} = 0.
\end{equation}
Thus, considering \eqref{quantities-cone} the solution of equations  \eqref{NewC-eq-1}  and
\eqref{NewC-eq-2}  gives
$C_{14} = C_{23} =0$ which is a contradiction. That is, $c_0=0$, and thus $x=0$.
Therefore $M$ is an open part of a timelike $zw$--plane.

\noindent {\bf Case 2.} $ h^3_{22} = - h^3_{11}\neq 0$, that is, 
$M$ is a pseudo--umbilical timelike surface in $\mathbb E_1^4$.
Now we will show that $M$ has no pointwise 1--type Gauss map of the second kind.
Note that for this case $h^4_{12} \not = 0$. 
If it were zero, then $M$ would be a cone obtained in Case 1
which is not pseudo--umbilical.
Similarly, considering \eqref{laplace-GaussMap} and using the Codazzi equation \eqref{CDYUzCodazzi1} 
we obtain the Laplacian of the Gauss map $\nu$ as
\begin{align} 
\label{laplace-GaussMap-pseud-U}
\Delta\nu =&||h||^2\nu +2 h_{12}^4 \omega_{34}(e_2) e_1\wedge e_4 + 2 h_{11}^3 \omega_{34}(e_2) e_2\wedge e_3.
\end{align}

Suppose that $M$ has pointwise 1--type Gauss map of the second kind.
Thus, \eqref{DefinitonPW1type} is satisfied
for some function $f\neq 0$ and nonzero constant vector $C$.
From \eqref{DefinitonPW1type}, \eqref{constantvectorC} and \eqref{laplace-GaussMap-pseud-U} 
we have
\begin{align}
\label{HHSbtRD0GaussTasCVek1}
f(1+C_{34})=&\|h\|^2,\\
\label{HHSbtRD0GaussTasCVek2}
fC_{14}  =& - 2 h^4_{12} \omega_{34}(e_2),\\
\label{HHSbtRD0GaussTasCVek3}
fC_{23} =&  2 h^3_{11} \omega_{34}(e_2),\\
\label{HHSbtRD0GaussTasCVek4}
C_{12}=&C_{13}=C_{24} =0.
\end{align}
From \eqref{HHSbtRD0GaussTasCVek2} and \eqref{HHSbtRD0GaussTasCVek3} 
it is seen that $C_{14} \not = 0$ and $C_{23} \not = 0$.
Equations \eqref{HHSbtRD0GaussTasCVek2} and \eqref{HHSbtRD0GaussTasCVek3} imply that
\begin{eqnarray}
\label{neweqC1}  h^3_{11}C_{14} +  h^4_{12} C_{23} =0.
\end{eqnarray}

From \eqref{MinkCVekSbtGYKos12} for $i=1$, we also obtain that
\begin{eqnarray}
\label{neweqC2}     h^4_{12} C_{14} - h^3_{11} C_{23} =0.
\end{eqnarray}
Hence, equations \eqref{neweqC1} and \eqref{neweqC2} give that $h^4_{12} = h^3_{11}  =0$, $(h^3_{22}  =0)$, 
that is, $M$ is an open part of the timelike $zw$--plane.

From Remark \ref{planeGaussmap}, the converse of the proof is trivial.
\end{proof}

\begin{cor}
There exists no a non--planar timelike rotational surface in $\mathbb E^4_1$ defined by 
\eqref{CiftDonelYuzeyMnfld} with flat normal bundle and pointwise 1--type Gauss map 
of the second kind.
\end{cor}

\begin{thm}
A timelike rotational surface with zero mean curvature and nonflat normal bundle 
in $\mathbb{E}^4_1$ defined by \eqref{CiftDonelYuzeyMnfld} has no pointwise 1--type Gauss map 
of the second kind.
\end{thm}

\begin{proof}
Let $M$ be a timelike rotational surface in $\mathbb{E}^4_1$ defined by \eqref{CiftDonelYuzeyMnfld}
and $\{e_1, e_2, e_3, e_4\}$ be an orthonormal moving frame on $M$ in $\mathbb{E}^4_1$ 
given by \eqref{TegNormVektorsCiftDY1}--\eqref{TegNormVektorsCiftDY4}.
Then the coefficients of the second fundamental form are given by 
\eqref{CDYuzeyTemelForm3} and \eqref{CDYuzeyTemelForm4}.
Since the mean curvature is zero and its normal bundle is nonflat,
from \eqref{CiftDonYuzOrtEgrVek} and \eqref{CiftDonYuzNormEgrilik} 
we have $h_{11}^3=h_{22}^3$ 
and $R^D(e_1,e_2;e_3,e_4)=h^4_{12}(h^3_{11}+h^3_{22})=2h^4_{12}h^3_{11}\neq 0$.
Hence the Laplacian of Gauss map $\nu=e_3\wedge e_4$ from \eqref{laplace-GaussMap} is given by
\begin{align} \label{laplace-GaussMap-minimal}
\Delta\nu =&||h||^2\nu - 4h_{12}^4h_{11}^3 e_1\wedge e_2.
\end{align}
We assume that M has pointwise 1--type Gauss map of the second kind. Therefore, from
\eqref{DefinitonPW1type}, \eqref{constantvectorC} and \eqref{laplace-GaussMap-minimal}
we have
\begin{align}
\label{HHSbtRD0GaussTasCVek1-1}
f(1+C_{34})=&\|h\|^2=2(h_{11}^3)^2 - 2(h_{12}^4)^2,\\
\label{HHSbtRD0GaussTasCVek2-2}
fC_{12}  =& 4 h^4_{12}h^3_{11} ,\\
\label{HHSbtRD0GaussTasCVek3-3}
C_{13}=&C_{14}=C_{23}= C_{24}= 0
\end{align}
from which we have $C_{12}\neq 0$. Considering \eqref{MinkCVekSbtGYKos13} and \eqref{MinkCVekSbtGYKos24}
for $i=2$ we obtain
$h^3_{22}C_{12} + h^4_{12}C_{34}=0$ and $h^4_{12}C_{12}-h^3_{22}C_{34}=0$.
The solution of these equations gives $C_{12}=0$ which is a contradiction.
Therefore, the Gauss map $\nu$ is not of pointwise 1--type Gauss map of the second kind.
\end{proof}

\section{Rotational surfaces in $\mathbb{E}^4_2$ with pointwise 1--type Gauss map}

In this section, we determine the rotational surfaces in the pseudo--Euclidean space $\mathbb{E}^4_2$ 
defined by \eqref{1DY} and \eqref{2DY} with pointwise 1--type Gauss map.

\begin{thm}
\label{M1pw2kind}
Let $M_1(b)$ be a non--planar regular rotational surface with zero mean curvature in $\mathbb{E}^4_2$ 
defined by \eqref{1DY}. Then,
\begin{itemize}
\item[i.] for some constants $\lambda_0\neq 0$ and $\mu_0$, the regular surface $M_1(1)$ 
with the profile curve $\alpha$ whose components satisfy 
\begin{equation}
\label{sol1}
(w(s)+y(s))^2+\lambda_0(w(s)-y(s))^2=\mu_0
\end{equation}
has pointwise 1--type Gauss map of the second kind. 

\item[ii.] for $b\neq 1$, the timelike surface $M_1(b)$ has pointwise 1--type Gauss map of the second kind 
if and only if the profile curve $\alpha$ is given by $y(s)=b_0(w(s))^{\pm b}$ for some constant $b_0\neq 0$.
\end{itemize}
\end{thm}

\begin{proof}
Assume that $M_1(b)$ is a non--planar regular rotational surface with zero mean curvature 
in $\mathbb{E}_2^4$ defined by \eqref{1DY}.
From equation \eqref{laplace-GaussMap}, the Laplacian of the Gauss map of 
the rotational surface $M_1(b)$ is given by
\begin{align}
\label{lapM1}
\Delta\nu=&||h||^2\nu+2h_{12}^4(\varepsilon^*h_{22}^3-\varepsilon h_{11}^3)e_1\wedge e_2\notag\\
&+\omega_{34}(e_1)(\varepsilon h_{11}^3+\varepsilon^*h_{22}^3)e_1\wedge e_3+(\varepsilon\varepsilon^*e_2(h_{11}^3)+e_2(h_{22}^3))e_2\wedge e_4.
\end{align}
Since the mean curvature of $M_1(b)$ is zero, equation \eqref{lapM1} becomes
\begin{equation}
\label{lapzeroM1}
\Delta\nu=||h||^2\nu-4\varepsilon h_{11}^3h_{12}^4 e_1\wedge e_2.
\end{equation} 
Suppose that $M_1(b)$ has pointwise 1--type Gauss map of second kind. 
Comparing \eqref{DefinitonPW1type} and \eqref{lapzeroM1}, we get 
\begin{align}
\label{eq1}
f(1+\varepsilon\varepsilon^*C_{34})&=||h||^2,\\
\label{eq2}
fC_{12}&=-4\varepsilon^*h_{11}^3h_{12}^4,\\
\label{eq3}
C_{13}&=C_{14}=C_{23}=C_{24}=0.
\end{align}
For $i=1, 2$, from \eqref{MinkCVekSbtGYKos13} and \eqref{MinkCVekSbtGYKos14}, we have
\begin{align}
\label{eq4}
h_{11}^3C_{12}+h_{12}^4C_{34}&=0,\\
\label{eq5}
h_{12}^4C_{12}+h_{11}^3C_{34}&=0.
\end{align}
Since the Gauss map $\nu$ is of the second kind, 
equations \eqref{eq4} and \eqref{eq5} must have nonzero solution 
which implies $(h_{11}^3)^2-(h_{12}^4)^2=0$. Considering the first equations in 
\eqref{1DYuzeyTemelForm3} and \eqref{1DYuzeyTemelForm4} 
we have $(b^2-1)(b^2y^2(s){w^\prime}^2(s)-w^2(s){y^\prime}^2(s))=0$, 
that is, $b=1$ or $b^2y^2(s){w^\prime}^2(s)-w^2(s){y^\prime}^2(s)=0$. 

If $b=1$, it was shown that the components of the profile curve $\alpha$ of the surface $M_1(1)$ 
with zero mean curvature satisfy equation \eqref{sol1}, \cite{Dursun-BB2}.
In this case, from \eqref{1DYuzeyTemelForm3} and \eqref{1DYuzeyTemelForm4} 
it can be seen easily that $h_{12}^4=-\varepsilon\varepsilon^* h_{11}^3$. 
Hence, by using equations \eqref{eq1}, \eqref{eq2} and
\eqref{eq4}, we find $C_{12}=-\frac{1}{2}$, $C_{34}=-\frac{\varepsilon\varepsilon^*}{2}$ 
and $f=-8\varepsilon (h_{22}^3)^2$.
Since $\alpha$ is a plane curve, $h_{22}^3=\kappa$, 
where $\kappa$ is a curvature of the curve $\alpha$. 
Thus, the Gauss map $\nu$ of $M_1(1)$ satisfies \eqref{DefinitonPW1type} 
for the function $f=-8\varepsilon\kappa^2$ and 
the constant vector $C=-\frac{\varepsilon\varepsilon^*}{2}e_1\wedge e_2-\frac{1}{2}e_3\wedge e_4$. 
This completes the proof of (a).

If $b^2y^2(s){w^\prime}^2(s)-w^2(s){y^\prime}^2(s)=0$ and $b\neq 1$, 
then we have $y(s)=b_0(w(s))^{\pm b}$,
where $b_0$ is nonzero constant. Also, the rotational surface $M_1(b)$ 
with this profile curve $\alpha$ is timelike, i.e., $\varepsilon\varepsilon^*=-1$. 
Hence, from the first equations in \eqref{1DYuzeyTemelForm3} and \eqref{1DYuzeyTemelForm4}, 
we get $h_{12}^4=\pm h_{11}^3$.
By using equations \eqref{eq1}, \eqref{eq2} and \eqref{eq4}, 
we get the function $f=-8\varepsilon\kappa^2$ 
and the constant vector $C=\pm\frac{1}{2}e_1\wedge e_2-\frac{1}{2}e_3\wedge e_4$. 

The converse of the proof is followed from a direct calculation. This completes the proof of (b).
\end{proof}

Similarly, we can state the following theorem for the rotational surface $M_2(b)$ 
defined by \eqref{2DY} in the pseudo--Euclidean space $\mathbb{E}^4_2$.

\begin{thm}
\label{M2pw2kind}
Let $M_2(b)$ be a non--planar regular rotational surface with zero mean curvature in $\mathbb{E}^4_2$ defined by \eqref{2DY}. Then,
\begin{itemize}
\item[i.] for some constants $\lambda_0\neq 0$ and $\mu_0$, 
the regular surface $M_2(1)$ with the profile curve $\beta$ whose components satisfy
\begin{equation}
\label{2sol1}
(x(s)+z(s))^2+\lambda_0(x(s)-z(s))^2=\mu_0
\end{equation}
has pointwise 1--type Gauss map of the second kind. 

\item[ii.] for $b\neq 1$,  the spacelike surface $M_2(b)$ has pointwise 1--type Gauss map of the second kind 
if and only if the profile curve $\beta$ is given by $z(s)=\bar{b}_0(x(s))^{\pm b}$ 
for some constant $\bar{b}_0\neq 0$.
\end{itemize}
\end{thm}

Note that considering equation $\eqref{lapzeroM1}$, 
if the Gauss map $\nu$ of the rotational surface $M_1(b)$ and $M_2(b)$ were of the first kind 
which implies that $h_{11}^3=0$ or $h_{12}^4=0$, 
then $M_1(b)$ and $M_2(b)$ would be lying in 3--dimensional pseudo--Euclidean space.

\begin{cor}
A rotational surface in the pseudo--Euclidean space $\mathbb{E}^4_2$ defined by \eqref{1DY} or \eqref{2DY} 
with zero mean curvature has no pointwise 1--type Gauss map of the first kind. 
\end{cor}


\begin{thebibliography}{HD}

\normalsize
\baselineskip=17pt

\bibitem{Dursun-BB1} {Bekta\c{s}, B. and Dursun, U.,}
\textit{Timelike Rotational Surfaces of Elliptic, Hyperbolic and Parabolic Types 
in Minkowski space $\mathbb{E}^4_1$ with Pointwise 1-Type Gauss Map}, Filomat, 29(2015), 381--392.


\bibitem{Dursun-BB2} {Bekta\c{s}, B., Canfes, E. and Dursun, U.,}
\textit{On Rotational Surfaces with Zero Mean Curvature in the Pseudo--Euclidean Space $\mathbb{E}^4_2$}, submitted.


\bibitem{CH1} {Chen, B.-Y.,}
\textit{Total Mean Curvature  and  Submanifolds of Finite Type}, World Scientific, Singapor-New Jersey-London, (1984).


\bibitem{Chen-Piccinni} {Chen, B.-Y. and  Piccinni, P.}
\emph{Submanifolds with Finite Type Gauss Map}, Bull. Austral. Math. Soc.,  {35}(1987), 161--186.


\bibitem{Choi-Kim-YoonDW2011} {Choi, M., Kim, Y.-H. and Yoon, D.W.},
\emph{Classification of Ruled Surfaces with Pointwise 1-Type Gauss Map in Minkowski 3-Space},
Taiwanese J. Math. {15}(2011), 1141-1161.


\bibitem{UDur2} {Dursun, U.},
\emph{Hypersurfaces with Pointwise 1-type Gauss Map in Lorentz-Minkowski Space}, 	
Proc. Est. Acad. Sci.,  {58}(2009), 146--161.


\bibitem{Dursun-BB} {Dursun, U. and Bekta\c{s}, B.},
\textit{Spacelike Rotational Surfaces of Elliptic, Hyperbolic and Parabolic Types 
in Minkowski space $\mathbb{E}^4_1$ with Pointwise 1-Type Gauss Map}, Mathematical Physics, Analysis and Geometry, 
{17}(2014), 247--263.


\bibitem{HuiLiGuiLi} {HuiLi, L. and GuiLi, L.}, 
\textit{Rotation Surfaces with Constant Mean Curvature in 4--Dimensional Pseudo--Euclidean Space}, Kyushu Journal of Mathematics, {48}(1994), 35--42.


\bibitem{KKKM} {Ki, U.H., Kim, D.S., Kim, Y.-H. and Roh, Y.M.},
\emph{Surfaces of Revolution with Pointwise 1-Type Gauss Map in Minkowski 3-Space},
Taiwanese J. Math.,  {13}(2009), 317--338.


\bibitem{Kim-Yoon} {Kim, Y.-H. and Yoon, D.W.},
\emph{Ruled Surfaces with Pointwise 1-Type Gauss Map},
J. Geom. Phys.,   {34}(2000), 191--205.


\bibitem{Kim-Yoon-2} {Kim, Y.-H. and Yoon, D.W.},
\emph{Classifications of Rotation Surfaces in Pseudo-Euclidean Space},
J. Korean  Math. Soc.,   {41}(2004), 379--396.


\bibitem{Kim-Yoon-3} {Kim, Y.-H. and Yoon, D.W.},
\emph{On the Gauss Map of Ruled Surfaces in Minkowski Space},
Rocky Mountain J. Math.,  {35}(2005),  1555--1581.

\end{thebibliography}
\end{document}